\documentclass[11pt]{amsart}

\usepackage{amsmath}
\usepackage{amssymb}
\usepackage{latexsym}
\usepackage{amscd}
\usepackage{mathrsfs}
\usepackage[all]{xy}

\newdimen\AAdi%
\newbox\AAbo%
\def\AAk#1#2{\s_etbox\AAbo=\hbox{#2}\AAdi=\wd\AAbo\kern#1\AAdi{}}%
\def\AAr#1#2#3{\s_etbox\AAbo=\hbox{#2}\AAdi=\ht\AAbo\raise#1\AAdi\hbox{#3}}%
\font\tenmsb=msbm10 at 12pt \font\sevenmsb=msbm7 at 8pt
\font\fivemsb=msbm5 at 6pt
\newfam\msbfam
\textfont\msbfam=\tenmsb \scriptfont\msbfam=\sevenmsb
\scriptscriptfont\msbfam=\fivemsb

\newtheorem{theorem}{Theorem}

\newtheorem{corollary}[theorem]{Corollary}

\newtheorem{lemma}[theorem]{Lemma}

\numberwithin{equation}{section} \numberwithin{theorem}{section}

\renewcommand{\topmargin}{0cm}
\renewcommand{\oddsidemargin}{5mm}
\renewcommand{\evensidemargin}{5mm}
\renewcommand{\textwidth}{150mm}
\renewcommand{\textheight}{230mm}

\def\R{\mathbb R}

\def\S{\mathbb S}

\def\na{\nabla}

\def\f#1#2{\frac{#1}{#2}}

\def\a{\alpha}

\def\r{\Re_{I\!V}}

\def\p#1{\partial #1}

\def\de{\delta}
\def\De{\Delta}
\def\e{\eta}
\def\ep{\epsilon}

\def\g{\gamma}
\def\k{\kappa}
\def\la{\lambda}
\def\La{\Lambda}
\def\lan{\langle}
\def\ran{\rangle}

\def\Om{\Omega}
\def\th{\theta}
\def\Th{\Theta}
\def\si{\sigma}
\def\Si{\Sigma}

\def\r{\rho}
\def\z{\zeta}

\begin{document}

\title
{A Bernstein type theorem for minimal hypersurfaces via Gauss maps}
\author{Qi Ding}
\address{Shanghai Center for Mathematical Sciences, Fudan University, Shanghai 200438, China}
\email{dingqi@fudan.edu.cn}

\thanks{The author is grateful to Tobias H. Colding for having invited him to visit the Department of Mathematics in MIT, where this work was done.
The author would like to thank J$\mathrm{\ddot{u}}$rgen Jost, Yuanlong Xin for their interest in this work.
The author would like to express his sincere gratitude to the referee for valuable comments that will help to improve the quality and accuracy of the manuscript.
He is partially sponsored by NSFC 11871156 and the China Scholarship Council.}
\date{}
\begin{abstract}
Let $M$ be an $n$-dimensional smooth oriented complete embedded minimal hypersurface in $\R^{n+1}$ with Euclidean volume growth. We show that if the image under the Gauss map of $M$ avoids some neighborhood of a half-equator, then $M$ must be an affine hyperplane.
\end{abstract}

\maketitle

\section{Introduction}

The original Bernstein theorem says that each entire minimal graph in $\R^3$ must be a plane.
The Bernstein theorem can be generalized to high dimensions as follows: each entire minimal graph in $\R^{n+1}$ must be a plane provided $n\le7$, which were achieved by successive
efforts of W. Fleming \cite{F}, E. De Giorgi \cite{DG}, F. J. Almgren \cite{Al}, and finally completely settled by J. Simons \cite{Si}.
For $n\ge8$, Bombieri-De Giorgi-Giusti provided a counterexample by constructing a nontrivial entire
minimal graph in $\R^{n+1}$, whose tangent cone at infinity is a vertical stable minimal cone, a non-warped product of a Simons' cone and line.
Under some conditions on graphic functions, all entire minimal graphs could be affine (see \cite{EH,M}).
In particular, all minimal graphs are stable minimal hypresurfaces.
In $\R^3$, all oriented complete stable minimal surfaces in $\R^3$ are affine plane shown by Fischer-Colbrie and Schoen \cite{FS}, and do Carmo-Peng \cite{DP}.
For $n\le 5$, with integral curvature estimates Schoen-Simon-Yau proved that all oriented complete stable minimal hypersurfaces with Euclidean volume growth in $\R^{n+1}$ must be affine \cite{SSY}. With the embedded condition, Schoen-Simon can show it for the case $n\le 6$ by their regularity theorem \cite{SS}.

Let $M$ be an $n$-dimensional smooth oriented complete minimal hypersurface in $\R^{n+1}$.
The Ruh-Vilms theorem tells us that the Gauss map $\g:\ M\rightarrow\S^n$ is a harmonic map \cite{RV}.
In \cite{So1}, Solomon showed that if $S$ is an area-minimizing hypersurface in $\R^{n+1}$ with $\p S=0$, the first
Betti number of reg$S$ vanishes and the Gauss map of $S$ omits some neighborhood of $\S^{n-2}$ in $\S^n$, then each component of spt$S$ is a hyperplane.
The condition on the first
Betti number is necessary by the example of Simons' cones (see also section 6 in \cite{So1} for instance).
In \cite{JXY}, Jost-Xin-Yang found a maximal open convex supporting subset $\S^n\setminus\overline{\S}^{n-1}_+$ of $\S^n$,
where $\overline{\S}^{n-1}_+$ is the hemisphere of $\S^{n-1}$ in $\S^n$.
They constructed a smooth bounded strictly convex function on any compact set $K$ in $\S^n\setminus\overline{\S}^{n-1}_+$, and then studied the regularity of harmonic maps to $K$.
As an application, they got a Bernstein type theorem as follows (see Theorem 6.5 in \cite{JXY}).
\begin{theorem}
Let $M^n\subset\R^{n+1}$ be a complete minimal embedded hypersurface with Euclidean volume growth.
Assume that there is a positive constant $C$, such that
for arbitrary $y\in M$ and $R>0$, the Neumann-Poincar$\mathrm{\acute{e}}$ inequality
$$\int_{M\cap \mathbf{B}_R(y)}|v-\bar{v}_{R,y}|^2\le CR^2\int_{M\cap \mathbf{B}_R(y)}|\na v|^2 $$
holds for each function $v\in C^\infty(\mathbf{B}_R(y))$, where $\mathbf{B}_R(y)$ is the ball in $\R^{n+1}$ with the radius $R$ and centered at $y$,
$\bar{v}_{R,y}$ is the average value of $v$ on $\mathbf{B}_R(y)$.
If the image under the Gauss map omits a neighborhood of $\overline{\S}^{n-1}_+$, then $M$ has to be
an affine linear space.
\end{theorem}

The necessity of the Neumann-Poincar$\mathrm{\acute{e}}$ inequality in the above theorem is not clear as they said in \cite{JXY}.
Later, Yang further proved that the above conclusion holds provided 1 of the distance between the Gauss image of $M\cap \mathbf{B}_R(y)$ and $\overline{\S}^{n-1}_+$ is less than $o(\log\log R)$ for the large $R$ in \cite{Y}.

In this paper, we remove the condition on the Neumann-Poincar$\mathrm{\acute{e}}$ inequality in Theorem 6.5 of \cite{JXY} instead by the oriented condition,
and obtain the following Bernstein type theorem.
\begin{theorem}
Let $M$ be an $n$-dimensional smooth oriented complete embedded minimal hypersurface in $\R^{n+1}$ with Euclidean volume growth. If the image under the Gauss map omits a neighborhood of $\overline{\S}^{n-1}_+$, then $M$ must be an affine hyperplane.
\end{theorem}
One of the important ingredients in the proof of our theorem is to show that the Gauss map of $M$ takes actually values in a compact subset of an open hemisphere of $\S^n$
provided the support of one of tangent cones of $M$ at infinity is the Euclidean space.

\section{New bounded subharmonic functions on minimal hypersurfaces}

Let $\mathbf{P}$ denote the projection from $\S^n$ onto $\overline{\mathbb{D}^2}$ (2-dimensional closed unit disk) by
$$\mathbf{P}:\ \S^n\rightarrow \overline{\mathbb{D}^2}\qquad\qquad (x_1,\cdots, x_{n+1})\mapsto (x_1,x_2).$$
For any $x=(x_1,\cdots,x_{n+1})\in\S^n\setminus\{(x_1,\cdots,x_{n+1})\in\S^{n}|\ x_1=0,\ x_2=0\}$,
there is a polar coordinate system in $\overline{\mathbb{D}^2}$. Namely,
$$\mathbf{P}(x)=\left(r(x)\sin\th(x),r(x)\cos\th(x)\right).$$
with $r(x)=\sqrt{x_1^2+x_2^2}\in(0,1]$ and the unique $\th(x)\in[\f12\pi,\f52\pi)$. In other words, we have defined two functions $r,\th$ on $\S^n\setminus\{(x_1,\cdots,x_{n+1})\in\S^{n}|\ x_1=0,\ x_2=0\}$.
Let $\si_\S$ be the standard metric on $\S^n$, and Hess be the Hessian matrix on $\S^n$ with the respect to $\si_\S$.
From \cite{JXY}, we have
\begin{equation}\aligned\label{Hessr}
\mathrm{Hess}\ r=-r\si_\S+rd\th\otimes d\th,
\endaligned
\end{equation}
and
\begin{equation}\aligned\label{Hessth}
\mathrm{Hess}\ \th=-r^{-1}\left(dr\otimes d\th+d\th\otimes dr\right).
\endaligned
\end{equation}
Let $\overline{\S}^{n-1}_+$ be a hemisphere of $\S^{n-1}$ defined by
\begin{equation}\aligned\label{+Sn-1}
\{(x_1,\cdots,x_{n+1})\in\S^{n}|\ x_1=0,\ x_2\ge0\},
\endaligned
\end{equation}
and $K$ be a compact set in $\S^n\setminus\overline{\S}^{n-1}_+$.
Clearly, there is a constant $\de_K>0$ such that $r(x)\ge\de_K$ for all $x\in K$.

Let $B_{\tau}(\S^{n-2})$ denote the $\tau$-neighborhood of $\{(0,0,x_3,\cdots,x_{n+1})|\ x_3^2+\cdots+x_{n+1}^2=1\}$ in $\S^n$, i.e.,
$$B_\tau(\S^{n-2})=\{(x_1,\cdots,x_{n+1})\in\S^{n}|\ x_1^2+x_{2}^2<\sin^2\tau\}.$$
We choose $0<\tau<1/2$ sufficiently small such that $B_{\tau}(\S^{n-2})\cap K=\emptyset$.
For each pair $x_*,-x_*\in \S^n\setminus B_{\tau}(\S^{n-2})$,
let $\th_*\in[\f12\pi,\f32\pi)$ be a constant such that $\th(x_*)=\th_*$, and $\th(-x_*)=\th_*+\pi$. Then we define a function
\begin{equation}\aligned\label{phik}
\phi=\f{\de_K}r+\f{k}2\left(\th-\th_*-\f{\pi}2\right)^2\qquad\qquad \mathrm{on}\ \S^n\setminus B_{\tau}(\S^{n-2}),
\endaligned
\end{equation}
where $k$ is an arbitrary constant $\ge1$.
In particular, $\phi(x_*)=\phi(-x_*)$, and $\phi$ is smooth on $K$.
Combining \eqref{Hessr} and \eqref{Hessth}, we have
\begin{equation}\aligned\label{Hessphi000}
\mathrm{Hess}\ \phi=&\f{\de_K}r\si_\S-\f{\de_K}rd\th\otimes d\th+\f{2\de_K}{r^3}dr\otimes dr\\
&-kr^{-1}\left(\th-\th_*-\f{\pi}2\right)\left(dr\otimes d\th+d\th\otimes dr\right)+kd\th\otimes d\th.
\endaligned
\end{equation}
\begin{lemma}\label{subhar}
Let $K$ be a compact set in $\S^n\setminus\overline{\S}^{n-1}_+$. For each pair $x_*,-x_*\in \S^n\setminus B_{\tau}(\S^{n-2})$ with $K\cap B_{\tau}(\S^{n-2})=\emptyset$ for some $\tau>0$, there is a bounded function $\Th$ on $\S^n\setminus B_{\tau}(\S^{n-2})$
with $\Th(x_*)=\Th(-x_*)$ such that
$\Th$ is smooth strictly convex on $K$.
\end{lemma}
\begin{proof}
Let $\phi$ be the function defined on $K$ as above with $\phi(x_*)=\phi(-x_*)$.
Enlightened by Lemma 2.1 in \cite{JXY}, we define $\Th=\la^{-1}e^{\la\phi}$ on $K$ with the positive constant $\la$ to be defined later.
Then $\phi(x_*)=\phi(-x_*)$ implies $\Th(x_*)=\Th(-x_*)$. By the definition of $\Th$, at any considered point $p\in\S^n$ we have
\begin{equation}\aligned\label{HessThxi}
\mathrm{Hess}\ \Th(\xi,\xi)=e^{\la\phi}\left(\mathrm{Hess}\ \phi(\xi,\xi)+\la|d\phi(\xi)|^2\right)
\endaligned
\end{equation}
for any unit vector $\xi\in T_p(\S^n)$. We denote $\tilde{\th}=\th-\th_*-\f{\pi}2$ for convenience. Combining \eqref{Hessphi000}, we have
\begin{equation}\aligned\label{laHessThxi}
e^{-\la\phi}\mathrm{Hess}\ \Th(\xi,\xi)=&\f{\de_K}r+\left(k-\f{\de_K}r\right)\left(d\th(\xi)\right)^2+\f{2\de_K}{r^3}\left(dr(\xi)\right)^2-\f{2k}{r}\tilde{\th}dr(\xi)d\th(\xi)\\
&+\la\left|\f{\de_K}{r^2}dr(\xi)-k\tilde{\th}d\th(\xi)\right|^2\\
=&\f{\de_K}r+\left(\la k^2\tilde{\th}^2+k-\f{\de_K}r\right)\left(d\th(\xi)\right)^2+\left(\f{2\de_K}{r^3}+\la\f{\de_K^2}{r^4}\right)\left(dr(\xi)\right)^2\\
&-2\left(\f{1}{r}+\la\f{\de_K}{r^2}\right)k\tilde{\th}dr(\xi)d\th(\xi).
\endaligned
\end{equation}
From Cauchy-Schwartz inequality, we have
\begin{equation}\aligned
2\left(\f{1}{r}+\la\f{\de_K}{r^2}\right)k\left|\tilde{\th}dr(\xi)d\th(\xi)\right|\le& \la k^2\tilde{\th}^2\left(d\th(\xi)\right)^2+\left(\f{1}{\sqrt{\la}r}+\sqrt{\la}\f{\de_K}{r^2}\right)^2\left(dr(\xi)\right)^2\\
=&\la k^2\tilde{\th}^2\left(d\th(\xi)\right)^2+\left(\f{1}{\la r^2}+\f{2\de_K}{r^3}+\la\f{\de_K^2}{r^4}\right)\left(dr(\xi)\right)^2.
\endaligned
\end{equation}
Substituting the above inequality into \eqref{laHessThxi} infers
\begin{equation}\aligned\label{HessTh***}
e^{-\la\phi}\mathrm{Hess}\ \Th(\xi,\xi)\ge&\f{\de_K}r+\left(k-\f{\de_K}r\right)\left(d\th(\xi)\right)^2-\f1{\la r^2}\left(dr(\xi)\right)^2\ge\f{\de_K}r-\f1{\la r^2}\left(dr(\xi)\right)^2
\endaligned
\end{equation}
as $k\ge1$.

In the upper hemisphere
\begin{equation}\aligned\label{UppSn-1}
\{(x_1,\cdots,x_{n+1})\in\S^{n}|\ x_{n+1}>0\},
\endaligned
\end{equation}
let $T$ denote a tensor defined by
\begin{equation}\aligned
T=\sum_{i,j=1}^n\f{x_ix_j}{1-\sum_{k=1}^n|x_k|^2}dx_i\otimes dx_j.
\endaligned
\end{equation}
Then the metric $\si_\S$ can be written as
\begin{equation}\aligned
\sum_{i=1}^ndx_i\otimes dx_i+T=dr\otimes dr+r^2d\th\otimes d\th+\sum_{i=3}^ndx_i\otimes dx_i+T.
\endaligned
\end{equation}
Since all the eigenvalues of $T$ are nonnegative, then $\si_\S\left(\f{\p}{\p r},\f{\p}{\p r}\right)\ge1$.
Hence at any point $q$ in \eqref{UppSn-1} and $q\in K$, we have
\begin{equation}\aligned\label{drxile1}
\sup_{\e\in T_q(\S^n)}|dr(\e)|\le|\e|.
\endaligned
\end{equation}
Similarly, \eqref{drxile1} holds at any $q\in \{(x_1,\cdots,x_{n+1})\in\S^{n}|\ x_{n+1}<0\}\cap K$. By the differentiability of $r$ on $\S^n\setminus\{(x_1,\cdots,x_{n+1})\in\S^{n}|\ x_1=0,\ x_2=0\}$, \eqref{drxile1} holds at any $q\in K$.
Therefore, if we choose $\la\ge\f2{\de_K^2}$, then from \eqref{HessTh***} we get
\begin{equation}\aligned
\mathrm{Hess}\ \Th(\xi,\xi)\ge e^{\la\phi}\f{\de_K}{2r}\qquad \qquad \mathrm{on} \ \ K.
\endaligned
\end{equation}
This completes the proof.
\end{proof}

Let $\De_{\S^n}$ and $\na_{\S^n}$ be the Laplacian and Levi-Civita connection on $\S^n$ with the respect to the metric $\si_\S$, respectively.
From Lemma \ref{subhar}, there is a constant $\k>0$ depending only on $n,K$ such that the Hessian of $\Th$ on $K\subset\S^n$ satisfies
\begin{equation}\aligned\label{subtv}
\mathrm{Hess}\,\Th\ge\k.
\endaligned
\end{equation}
By the construction of $\Th$, there is a constant $c_K$ depending only on $n,K$ such that
\begin{equation}\aligned\label{vnu}
|\Th(\nu)-\Th(\nu')|\le c_K|\nu-\nu'|
\endaligned
\end{equation}
for any $\nu,\nu'\in K$ from Newton-Leibnitz formula.

Let $M$ be an $n$-dimensional complete oriented smooth minimal hypersurface in $\R^{n+1}$.
Namely, there are a smooth $n$-dimensional Riemannian manifold $M'$, and an isometric mapping $X:\ M'\rightarrow M$ with $X(M')=M\subset\R^{n+1}$.
Let $\g: M'\rightarrow\S^n$ be the Gauss map defined by $\g(p)=X_*(T_pM')\in\S^n$ via the parallel translation in $\R^{n+1}$ for all $p\in M'$.
For convenience, we identify $M$ and $M'$ by viewing $X(p)$ as $p$.
Suppose that the Gauss image of $M$ satisfies $\g(M)\subset K$, Then we define the function $v=\Th\circ\g$ on $M$.
Let $\na$ be the Levi-Civita connection of $M$ with the induced metric from $\R^{n+1}$.
We choose a local orthonormal tangent frame $\{e_i\}$ on $M$ such that $\na e_i=0$ at the considered point. Then (see formula (2.9) in \cite{DXY} for instance)
\begin{equation}\aligned\label{|A|2***}
\sum_{i=1}^n\lan\g_*e_i,\g_*e_i\ran=|A|^2.
\endaligned
\end{equation}
As $\g$ is a harmonic map from $M$ to $\S^n$, combining \eqref{subtv} we get
\begin{equation}\aligned\label{subv}
\De_M v=\sum_{i=1}^n\mathrm{Hess}\, \Th(\g_*e_i,\g_*e_i)\ge \k\sum_{i=1}^n\lan\g_*e_i,\g_*e_i\ran=\k|A|^2,
\endaligned
\end{equation}
where $\De_{M}$ is the Laplacian of $M$ with the induced metric from $\R^{n+1}$, $A$ is the second fundamental form of $M$. In particular, for any $x_*,-x_*\in\overline{\g(M)}\subset K$, we can assume $\Th(x_*)=\Th(-x_*)$.

We say that $M$ is \emph{$\de$-stable} if
\begin{equation}\aligned\label{destable}
\int_{M}\left(|\na f|^2-\de|A|^2f^2\right)\ge0
\endaligned
\end{equation}
for each smooth function $f: M\rightarrow\R$ with compact support.
\begin{lemma}\label{dekstable}
Let $M$ be an $n$-dimensional complete oriented minimal hypersurface in $\R^{n+1}$. If there is a bounded function $v$ on $M$ satisfying \eqref{subv}
for some positive number $\k>0$, then $M$ is $\de_\k$-stable for some $\de_\k>0$.
\end{lemma}
\begin{proof}
Let $v_{sup}$ and $v_{inf}$ be the supremum and infimum of $v$ on $M$, respectively. From \eqref{subv}, one has
\begin{equation}\aligned
\k|A|^2\le-\De_M(v_{sup}-v)\le-\f{v_{sup}-v_{inf}}{v_{sup}-v}\De_M(v_{sup}-v).
\endaligned
\end{equation}
Taking $\de_\k=\f{\k}{v_{sup}-v_{inf}}$, $\tilde{v}=v_{sup}-v$, then
\begin{equation}\aligned
\de_\k|A|^2\le-\f{1}{\tilde{v}}\De_M\tilde{v}.
\endaligned
\end{equation}
For any smooth function $\phi$ on $M$ with compact support, we have (see also the proof of proposition 6.2.2 of \cite{X})
\begin{equation}\aligned
&\int_M\left(|\na\phi|^2-\de_\k|A|^2\phi^2\right)\ge\int_M\left(|\na\phi|^2+\f{\phi^2}{\tilde{v}}\De_M\tilde{v}\right)\\
=&\int_M\left(|\na\phi|^2-\na\tilde{v}\cdot\na\left(\f{\phi^2}{\tilde{v}}\right)\right)
=\int_M\left(|\na\phi|^2-2\f{\phi}{\tilde{v}}\na\tilde{v}\cdot\na\phi+\f{\phi^2}{\tilde{v}^2}|\na\tilde{v}|^2\right)\ge0,
\endaligned
\end{equation}
where we have used Cauchy-Schwartz inequality in the last step. Namely, $M$ is a smooth $\de_\k$-stable minimal hypersurface in $\R^{n+1}$.
\end{proof}
As a corollary, an $n$-dimensional complete oriented minimal hypersurface $M$ in $\R^{n+1}$ is $\de$-stable for some $\de>0$
provided the Gauss image of $M$ is contained in $K$, where $K$ is defined as above.

\section{Rigidity of minimal hypersurfaces with constraints}

Let $\mathbf{B}_r(X)$ denote the ball in $\R^{n+1}$ with radius $r>0$ and centered at $X\in\R^{n+1}$. Denote $\mathbf{B}_r=\mathbf{B}_r(0)$ for simplicity.
Let $M$ be a smooth oriented complete minimal hypersurface in $\R^{n+1}$ with Euclidean volume growth and $\g(M)\subset K$,
where $K$ is a compact set in $\S^n\setminus\overline{\S}^{n-1}_+$ and $\overline{\S}^{n-1}_+$ is defined as \eqref{+Sn-1}.
Now we assume that the support of one of tangent cones of $M$ at infinity is an $n$-Euclidean space.
Namely, there is a sequence $R_i\rightarrow\infty$ such that $\f1{R_i}M$ converges to an integer varifold $T$ with the support spt$T$ being the $n$-Euclidean space.
Let $\nu_*,-\nu_*$ denote the unit normal vectors of spt$T$.
From the definition of the function $\th$ in the last chapter, without loss of generality, we assume that $\th(\nu_*)\in[\f12\pi,\f32\pi)$.
Let $\nu$ denote the unit normal vector field of $M$.
From \cite{SS} (or 22.2 in \cite{S}), the \emph{unoriented excess} satisfies
\begin{equation}\aligned
\lim_{i\rightarrow\infty}R_i^{-n}\int_{\mathbf{B}_{R_i}\cap M}\left(1-\lan\nu,\nu_*\ran^2\right)=0.
\endaligned
\end{equation}
It is easy to see
\begin{equation}\aligned\label{inenunu*}
\min\{1-\lan\nu,\nu_*\ran,1+\lan\nu,\nu_*\ran\}\le1-\lan\nu,\nu_*\ran^2.
\endaligned
\end{equation}
With Cauchy inequality, we have
\begin{equation}\aligned\label{asympnu*}
\lim_{i\rightarrow\infty}R_i^{-n}\int_{\mathbf{B}_{R_i}\cap M}\min\{|\nu-\nu_*|,|\nu+\nu_*|\}=0.
\endaligned
\end{equation}
From Lemma \ref{subhar},
there is a bounded function $\Th$ on $\S^n\setminus B_{\tau}(\S^{n-2})$ with $\Th(\nu_*)=\Th(-\nu_*)$ and $B_{\tau}(\S^{n-2})\cap K=\emptyset$
such that
$\Th$ is smooth strictly convex on $K$. Put $\g:\ M\rightarrow\S^n$ be the Gauss map and $v=\Th\circ\g$ as before.
Denote $v_*=\Th(\nu_*)$.

\begin{lemma}\label{Asympv*}
We have
\begin{equation}\aligned\label{asympv*}
\lim_{i\rightarrow\infty}R_{i}^{-n}\int_{\mathbf{B}_{R_{i}}\cap M}|v-v_*|=0.
\endaligned
\end{equation}
\end{lemma}
\begin{proof}
Let us prove it by dividing into 3 cases.

Case 1: $\nu_*,-\nu_*\in\overline{\g(M)}\subset K$.
From \eqref{vnu}, we have
\begin{equation}\aligned\label{vnu*}
|\Th(\nu)-\Th(\nu_*)|=|\Th(\nu)-\Th(-\nu_*)|\le c_K\min\{|\nu-\nu_*|,|\nu+\nu_*|\}.
\endaligned
\end{equation}
Combining \eqref{asympnu*} and \eqref{vnu*}, we can get \eqref{asympv*}.

Case 2: $-\nu_*\in K$ and $\nu_*\in\S^n\setminus K$, then \eqref{asympnu*} infers
\begin{equation}\aligned\label{asympnu*2}
\lim_{i\rightarrow\infty}R_i^{-n}\int_{\mathbf{B}_{R_i}\cap M}|\nu+\nu_*|=0.
\endaligned
\end{equation}
From \eqref{vnu} and $\Th(\nu_*)=\Th(-\nu_*)$, one has $|\Th(\nu)-\Th(\nu_*)|\le c_K|\nu+\nu_*|$ for any $\nu\in K$.
Combining \eqref{asympnu*2}, we can get \eqref{asympv*}.

Case 3: $\nu_*\in K$ and $-\nu_*\in\S^n\setminus K$, then \eqref{asympnu*} infers
\begin{equation}\aligned\label{asympnu*1}
\lim_{i\rightarrow\infty}R_i^{-n}\int_{\mathbf{B}_{R_i}\cap M}|\nu-\nu_*|=0.
\endaligned
\end{equation}
Combining \eqref{vnu} and \eqref{asympnu*1}, we can get \eqref{asympv*} analog to the case 2, and complete the proof.
\end{proof}

\begin{lemma}
The supermum of the function $v$ on $M$ is $v_*$.
\end{lemma}
\begin{proof}
From Euclidean volume growth of $M$ and monotonicity of $\r^{-n}\mathcal{H}^n\left(M\cap\mathbf{B}_{\r}(p)\right)$,
$M\cap\mathbf{B}_{\r}(p)$ has volume doubling condition (independent of $p,\r$).
Combining Sobolev inequality on $M$ (see \cite{S}),
there is a positive constant $c_M$ depending only on $n$ and the limit of $\r^{-n}\mathcal{H}^n\left(M\cap\mathbf{B}_{\r}(p)\right)$
such that for each $p\in M$, each smooth nonnegative subharmonic function $f$ on $M$, one has the mean value inequality
\begin{equation}\aligned\label{subf}
f(p)\le\f{c_M}{\r^n}\int_{M\cap\mathbf{B}_{\r}(p)} f.
\endaligned
\end{equation}
Now let us prove the lemma by contradiction. Namely, we assume $\sup_Mv>v_*$. For each $0<\ep<\f12\left(\sup_Mv-v_*\right)$, we define
$$v_\ep=\left(\sup_Mv+\ep-v\right)^{-1}.$$
Then
\begin{equation}\aligned
\De_M v_\ep\ge v_\ep^2\De_M v\ge0.
\endaligned
\end{equation}
There is a point $p\in M$ such that $v(p)\ge\sup_Mv-\ep$. From \eqref{subf}, we have
\begin{equation}\aligned\label{vepvep}
&\f1{2\ep}\le v_\ep(p)\le\f{c_M}{\r^n}\int_{M\cap\mathbf{B}_{\r}(p)}v_\ep\\
=&\f{c_M}{\r^n}\int_{M\cap\mathbf{B}_{\r}(p)}\left(\f1{\sup_Mv+\ep-v}-\f1{\sup_Mv+\ep-v_*}+\f1{\sup_Mv+\ep-v_*}\right)\\
\le&\f{c_M}{\r^n}\int_{M\cap\mathbf{B}_{\r}(p)}\left(\f{v-v_*}{\left(\sup_Mv+\ep-v\right)\left(\sup_Mv+\ep-v_*\right)}+\f1{\sup_Mv-v_*}\right)\\
\le&\f{c_M}{\ep^2\r^n}\int_{M\cap\mathbf{B}_{\r}(p)}|v-v_*|+\f{c_M}{\sup_Mv-v_*}\r^{-n}\mathcal{H}^n\left(M\cap\mathbf{B}_{\r}(p)\right).
\endaligned
\end{equation}
Hence combining \eqref{asympv*}, one has
\begin{equation}\aligned
\f1{2\ep}\le&\lim_{R_i\rightarrow\infty}\f{c_M}{\ep^2R_i^n}\int_{M\cap\mathbf{B}_{R_i}(p)}|v-v_*|
+\f{c_M}{\sup_Mv-v_*}\lim_{\r\rightarrow\infty}\r^{-n}\mathcal{H}^n\left(M\cap\mathbf{B}_{\r}(p)\right)\\
=&\f{c_M}{\sup_Mv-v_*}\lim_{\r\rightarrow\infty}\r^{-n}\mathcal{H}^n\left(M\cap\mathbf{B}_{\r}(p)\right).
\endaligned
\end{equation}
The above inequality fails for the sufficiently small $\ep>0$. Hence we complete the proof.
\end{proof}

For each small $t>0$, we define a closed set $E_t\in\overline{\mathbb{D}^2}$ by
$$E_t=\{(r\sin\th,r\cos\th)|\ r\in[\de_K,1],\ \th_*-t\le\th\le\th_*+\pi+t\}.$$
We define a closed set $\widetilde{E}_t\subset\S^n$ as the inverse image of $E_t$ under the mapping $\mathbf{P}$, i.e.,
\begin{equation}\aligned\label{EtEt}
\widetilde{E}_t=\mathbf{P}^{-1}(E_t).
\endaligned
\end{equation}
From $v_*=\sup_M v$, we get $\Th(\nu_*)=\Th(-\nu_*)=\sup_{\g(M)}\Th$. With the definition of $\phi$ in \eqref{phik},
for each small $t>0$ we have $\g(M)\subset\widetilde{E}_t$ when $k\ge2$ is sufficiently large, which implies
$$\g(M)\subset\lim_{t\rightarrow0}\widetilde{E}_t.$$
In particular, $\g(M)$ is contained in a closed hemi-sphere of $\S^n$, denoted by
$$\{\xi\in\S^n|\ \lan\xi,\nu_0\ran\ge0\}$$
for the unique unit vector $\nu_0\in\S^n$ with $\mathbf{P}(\nu_0)=(\sin(\th(\nu_*)+\pi/2),\cos(\th(\nu_*)+\pi/2))$.
From the well-known formula (see formula (1.3.8) in \cite{X} for instance)
\begin{equation}\aligned\label{nunu0}
\De_M\lan\nu,\nu_0\ran=-|A|^2\lan\nu,\nu_0\ran,
\endaligned
\end{equation}
and the strong maximum principle, we have $\lan\nu,\nu_0\ran>0$ on $M$,
or $\lan\nu,\nu_0\ran\equiv0$ on $M$.
However, the later case only occurs for $M$ being a Euclidean space. Now we consider the case $\lan\nu,\nu_0\ran>0$ on $M$. Hence, it follows that
$$\g(M)\subset K\cap\{\xi\in\S^n|\ \lan\xi,\nu_0\ran>0\}.$$
From \eqref{nunu0} and the argument of Lemma \ref{dekstable}, $M$ is a stable minimal hypersurface in $\R^{n+1}$.
With \cite{SSY}, we have gotten the flatness of $M$ for $n\le 5$. For general $n$, let us prove the flatness by constructing new bounded subharmonic functions on $M$.

Recall
$$B_\tau(\S^{n-2})=\{(x_1,\cdots,x_{n+1})\in\S^{n}|\ x_1^2+x_{2}^2<\sin^2\tau\}.$$
We fix the sufficiently small constant $\tau\in(0,1/2)$ such that $B_\tau(\S^{n-2})\cap K=\emptyset$.
Let $\Om$ be the closed set in $\S^n$ defined by
\begin{equation}\aligned\label{Om}
\{\xi\in\S^n|\ \lan\xi,\nu_0\ran\ge0\}\setminus B_\tau(\S^{n-2}).
\endaligned
\end{equation}
Then $\g(M)\subset\Om.$
\begin{lemma}\label{HessLa***}
For any positive constant $\tau_n<1$, there is a positive smooth function $\La$ on $\Om$ such that
\begin{equation}\aligned\label{HessLa}
\mathrm{Hess}\ \La\le-(1-\tau_n)\La\si_{\S}\qquad \qquad \mathrm{on}\ \Om.
\endaligned
\end{equation}
\end{lemma}
\begin{proof}
By the choice of coordinates of $\S^n$ and the definition of $\nu_0$, (in this proof) we can allow
$$\Om= \{(x_1,\cdots,x_{n+1})\in\S^{n}|\ x_{1}\ge0\}\setminus B_\tau(\S^{n-2}).$$
Now we define a smooth cut-off function $\e$ on $(-\infty,1]$ by $\e\equiv1$ on $(-\infty,\tau/4]$, $\e\equiv0$ on $[\tau/2,1]$, and $0\le\e\le1$, $|\e'|\le c/\tau$, $|\e''|\le c/\tau^2$.
Here, $c$ is an absolute positive constant. For each $\ep>0$, we define a positive function $\varphi_\ep$ on $\Om$ by
$$\varphi_\ep=x_{1}+\ep\e(x_{1})|x_2|.$$
Since $\sin\tau\ge\f2{\pi}\tau$ on $[0,\f{\pi}2]$, then obviously $\varphi_\ep$ is well-defined and smooth on $\Om$. Put $\Om^+_t=\Om\cap\{x_2>0\}\cap\{0\le x_{1}<t\}$ and $\Om^-_t=\Om\cap\{x_2<0\}\cap\{0\le x_{1}<t\}$ for each $0<t<\sin\tau$.
From \cite{JXY}, $\mathrm{Hess}\ x_i=-x_i\si_\S$ on $\S^n$ for each $1\le i\le n+1$. Hence there is an absolute positive constant $c'>0$ such that
\begin{equation}\aligned\label{Hessex1x2}
\left|\mathrm{Hess}(\e(x_{1}) x_2)\right|\le\f{c'}{\tau^2}.
\endaligned
\end{equation}
On $\Om^+_{\tau/4}$ one has
\begin{equation}\aligned
\mathrm{Hess}\ \varphi_\ep=\mathrm{Hess}\ x_{1}+\ep\mathrm{Hess}\ x_2=-x_{1}\si_\S-\ep x_{2}\si_\S=-\varphi_\ep\si_\S,
\endaligned
\end{equation}
and similarly on $\Om^-_{\tau/4}$ one has
\begin{equation}\aligned
\mathrm{Hess}\ \varphi_\ep=\mathrm{Hess}\ x_{1}-\ep\mathrm{Hess}\ x_2=-x_{1}\si_\S+\ep x_{2}\si_\S=-\varphi_\ep\si_\S.
\endaligned
\end{equation}
Note that $0<\tau<1/2$ is fixed. For each $0<\tau_n<1$, we choose $\ep$ sufficiently small such that $\f{c'\ep}{\tau^2}<\f{\tau}4\tau_n-\ep$.
On $\Om^+_{\tau/2}\setminus\Om^+_{\tau/4}$, combining \eqref{Hessex1x2} we have
\begin{equation}\aligned
\mathrm{Hess}\ \varphi_\ep=&\mathrm{Hess}\ x_{1}+\ep\mathrm{Hess}(\e x_2)\le-x_{1}\si_\S+\f{c'\ep}{\tau^2}\si_\S\\
\le&\left(-x_{1}+\f{\tau}4\tau_n-\ep\right)\si_\S\le\left(-x_{1}+\tau_nx_{1}-\ep\right)\si_\S\le-(1-\tau_n)\varphi_\ep\si_\S,
\endaligned
\end{equation}
and similarly on $\Om^-_{\tau/2}\setminus\Om^-_{\tau/4}$ one has
\begin{equation}\aligned
\mathrm{Hess}\ \varphi_\ep=\mathrm{Hess}\ x_{1}-\ep\mathrm{Hess}(\e x_2)\le-x_{1}\si_\S+\f{c'\ep}{\tau^2}\si_\S\le-(1-\tau_n)\varphi_\ep\si_\S.
\endaligned
\end{equation}
This is sufficient to complete the proof.
\end{proof}
Note that $\La$ in Lemma \ref{HessLa***} is strictly bounded away from zero as $\Om$ is closed.
Put $\Phi=\f1{\La\circ\g}$, then it is a smooth positive bounded function on $M$.
Let $\na$ be the Levi-Civita connection of $M$ as before.
We choose a local orthonormal tangent frame $\{e_i\}$ on $M$ such that $\na e_i=0$ at the considered point.
As $\g$ is a harmonic map from $M$ to $\S^n$, combining \eqref{|A|2***} and Lemma \ref{HessLa***} we conclude
\begin{equation}\aligned\label{DePhi*}
\De_M\Phi^{-1}=\sum_{i=1}^n\mathrm{Hess}\ \La(\g_*e_i,\g_*e_i)\le -(1-\tau_n)\Phi^{-1}\sum_{i=1}^n\lan\g_*e_i,\g_*e_i\ran=-(1-\tau_n)\Phi^{-1}|A|^2.
\endaligned
\end{equation}

\begin{theorem}\label{BerOm}
Let $\Om$ be the set defined in \eqref{Om}.
Let $M$ be a smooth oriented complete minimal hypersurface in $\R^{n+1}$ with Euclidean volume growth and $\g(M)\subset \Om$. Then $M$ is an affine hyperplane in $\R^n$.
\end{theorem}
\begin{proof}
From \eqref{DePhi*}, there is a smooth positive bounded function $\Phi$ on $M$ satisfying
\begin{equation}\aligned
\De_M\Phi^{-1}\le -\left(1-\f1{2n}\right)\Phi^{-1}|A|^2\qquad\qquad \mathrm{on}\ M.
\endaligned
\end{equation}
Namely,
\begin{equation}\aligned\label{DeP}
\De_M\Phi\ge \left(1-\f1{2n}\right)\Phi|A|^2+2\Phi^{-1}|\na\Phi|^2\qquad\qquad \mathrm{on}\ M.
\endaligned
\end{equation}
Recall Simons' inequality \cite{Si} (see also the formula (4) in \cite{EH} for instance):
\begin{equation}\aligned\label{DeA}
\De_M|A|^2\ge -2|A|^4+2\left(1+\f2{n}\right)|\na|A||^2\qquad\qquad \mathrm{on}\ M.
\endaligned
\end{equation}
Now we can use the idea in \cite{EH} by Ecker-Huisken to show the flatness of $M$. For any positive constants $p,q>0$, from \eqref{DeP}\eqref{DeA} one has
\begin{equation}\aligned
\De_M\left(|A|^p\Phi^q\right)\ge& \left(\left(1-\f1{2n}\right)q-p\right)|A|^{p+2}\Phi^q+p\left(p-1+\f2n\right)|A|^{p-2}\Phi^q|\na|A||^2\\
&+q(q+1)|A|^p\Phi^{q-2}|\na\Phi|^2+2pq|A|^{p-1}\Phi^{q-1}\lan\na|A|,\na\Phi\ran.
\endaligned
\end{equation}
With Young's inequality we derive
\begin{equation}\aligned
\De_M\left(|A|^p\Phi^q\right)\ge \left(\left(1-\f1{2n}\right)q-p\right)|A|^{p+2}\Phi^q+p\left(\f{p}{q+1}-1+\f2n\right)|A|^{p-2}\Phi^q|\na|A||^2.
\endaligned
\end{equation}
For $p=n,q=n+1$, we have
\begin{equation}\aligned\label{APhi}
\De_M\left(|A|^n\Phi^{n+1}\right)\ge \left(\f12-\f1{2n}\right)|A|^{n+2}\Phi^{n+1}.
\endaligned
\end{equation}
For $p=2n+2,q=2n+4$, we have
\begin{equation}\aligned\label{APhi*}
\De_M\left(|A|^{2n+2}\Phi^{2n+4}\right)\ge 0.
\endaligned
\end{equation}

Let $\Phi_*$ be a positive constant $>1$ so that $\Phi_*^{-1}<\Phi<\Phi_*$ on $M$.
From \eqref{subf} and \eqref{APhi*}, for each $z\in M$ we have
\begin{equation}\aligned\label{AnPhin+1}
|A|^{2n+2}\Phi^{2n+4}(z)\le\f{c_M}{\r^n}\int_{M\cap\mathbf{B}_{\r}(z)} |A|^{2n+2}\Phi^{2n+4}\le\f{c_M\Phi_*^{2}}{\r^n}\int_{M\cap\mathbf{B}_{\r}(z)} |A|^{2n+2}\Phi^{2n+2}.
\endaligned
\end{equation}
Let $\z$ be a smooth positive function in $[0,\infty)$ by $\z(r)=1$ for $0\le r\le \r$,
$\z(r)=0$ for $r\ge2\r$,
and $|\z'|\le c_n\r^{-1}$ for $\r\le r\le 2\r$.
Here, $c_n$ is a positive constant depending only on $n$.
We multiply \eqref{APhi} on both sides by $|A|^n\Phi^{n+1}\z^{2n+2}(|X|)$, and integrate by parts in conjunction with Young's inequality, then
\begin{equation}\aligned
&\left(\f12-\f1{2n}\right)\int_{M} |A|^{2n+2}\Phi^{2n+2}\z^{2n+2}\le\int_M|A|^n\Phi^{n+1}\z^{2n+2}\De_M\left(|A|^n\Phi^{n+1}\right)\\
=&-\int_M\left|\na\left(|A|^n\Phi^{n+1}\right)\right|^2\z^{2n+2}-2(n+1)\int_M|A|^n\Phi^{n+1}\z^{2n+1}\left\lan\na\z,\na\left(|A|^n\Phi^{n+1}\right)\right\ran\\
\le&(n+1)^2\int_M|A|^{2n}\Phi^{2n+2}\z^{2n}|\na\z|^2\\
\le&(n+1)^2\left(\int_{M} |A|^{2n+2}\Phi^{2n+2}\z^{2n+2}\right)^{\f n{n+1}}\left(\int_{M}\Phi^{2n+2}|\na\z|^{2n+2}\right)^{\f 1{n+1}}.
\endaligned
\end{equation}
Note $n\ge2$. The above inequality implies
\begin{equation}\aligned\label{APhiz}
\int_{M} |A|^{2n+2}\Phi^{2n+2}\z^{2n+2}\le2^{2n+2}(n+1)^{2n+2}\int_{M}\Phi^{2n+2}|\na\z|^{2n+2}.
\endaligned
\end{equation}
Combining \eqref{AnPhin+1}\eqref{APhiz} and $\Phi_*^{-1}<\Phi<\Phi_*$, $|\z'|\le c_n\r^{-1}$, we have
\begin{equation}\aligned
&\Phi_*^{-2n-4}|A|^{2n+2}(z)\le\f{c_M\Phi_*^{2}}{\r^n}\int_{M} |A|^{2n+2}\Phi^{2n+2}\z^{2n+2}\\
\le& 2^{2n+2}(n+1)^{2n+2}c_M\Phi_*^{2n+4}\r^{-n}\int_{M} |\na\z|^{2n+2}\\
\le& 2^{2n+2}(n+1)^{2n+2}c_Mc_n^{2n+2}\Phi_*^{2n+4}\r^{-3n-2}\mathcal{H}^n(M\cap B_{2\r}(z)\setminus B_\r(z)).
\endaligned
\end{equation}
Letting $\r\rightarrow\infty$, we get $|A|=0$ at $z$, which completes the proof.
\end{proof}


In all, we have proven the following rigidity result in this chapter.
\begin{corollary}\label{Multione}
Let $M$ be a smooth oriented complete minimal hypersurface in $\R^{n+1}$ with Euclidean volume growth and the Gauss image $\g(M)\subset K$, where $K$ is a compact set in $\S^n\setminus\overline{\S}^{n-1}_+$. If the support of one of tangent cones of $M$ at infinity is the Euclidean space, then $M$ is an affine linear space.
\end{corollary}

\section{Regularity of minimal hypersurfaces}

Let $M$ be a smooth oriented minimal hypersurface in $\mathbf{B}_{2\r}(0)\subset\R^{n+1}$ with $\p M\subset\p\mathbf{B}_{2\r}(0)$ and the Gauss image $\g(M)\subset K$, where $K$ is a compact set in $\S^n\setminus\overline{\S}^{n-1}_+$.
Moreover, we assume
\begin{equation}\aligned\label{ASSArea}
\mathcal{H}^n(M\cap \mathbf{B}_{2\r}(0))<\a\r^n
\endaligned
\end{equation}
for some constant $\a>0$.
For each $X_0\in M\cap\mathbf{B}_{\r}(0)$ and each $0<\r_1\le\r$, by the monotonicity formula, we have
\begin{equation}\aligned
\r_1^{-n}\mathcal{H}^n(M\cap \mathbf{B}_{\r_1}(X_0))\le\r^{-n}\mathcal{H}^n(M\cap \mathbf{B}_{\r}(X_0))<\r^{-n}\mathcal{H}^n(M\cap \mathbf{B}_{2\r}(0))\le\a.
\endaligned
\end{equation}
For simplicity, we denote $\mathbf{B}_r=\mathbf{B}_r(0)$ for all $r>0$.
Now let us use the technique in the last chapter to show the following regularity theorem.
\begin{theorem}\label{M1Reg}
Let $K,\a$ be as above.
There is a positive constant $\de_{K,\a}$ depending only on $n,K,\a$ such that if
\begin{equation}\aligned\label{Ass1}
\inf_{\xi\in\S^n}\sup_{X\in M\cap\mathbf{B}_{2\r}}|\lan X,\xi\ran|<\de_{K,\a}\r,
\endaligned
\end{equation}
then $|A|\le 1/\r$ on $M\cap\mathbf{B}_{\r/2}$. Here, $A$ is the second fundamental form of $M$.
\end{theorem}
\begin{proof}
Without loss of generality, we can assume that $M$ is a connected smooth manifold.
There is a small constant $\tau_K>0$ depending only on $K$ such that $B_{\tau_K}(K)$, the $\tau_K$-neighborhood of $K$ in $\S^n$, is still contained in $\S^n\setminus\overline{\S}^{n-1}_+$.
Let $\widehat{K}$ denote $B_{\tau_K/2}(K)$, the $\tau_K/2$-neighborhood of $K$ in $\S^n$.
Let $\Psi(t|E,\a)$ denote a general positive function on $(0,1)\times \left(\S^n\setminus\overline{\S}^{n-1}_+\right)\times(0,\infty)$ with $\lim_{t\rightarrow0}\Psi(t|E,\a)=0$ for each $E\in \S^n\setminus\overline{\S}^{n-1}_+$ and $\a>0$.
For simplicity, we fix $(\widehat{K},\a)\in \left(\S^n\setminus\overline{\S}^{n-1}_+\right)\times(0,\infty)$, and assume $\Psi(t)=\Psi(t|\widehat{K},\a)$. We allow that $\Psi(t)$ changes from the line to the line.

We assume that there are a unit vector $\xi\in\S^n$ and a constant $\de>0$ such that $|\lan X,\xi\ran|\le\de\r$ for each $X\in M\cap\mathbf{B}_{2\r}$.
From \cite{SS} (or 22.2 in \cite{S}), the \emph{unoriented excess} satisfies
\begin{equation}\aligned
\int_{\mathbf{B}_{\f32\r}\cap M}\left(1-\lan\nu,\xi\ran^2\right)<\Psi(\de)\r^n,
\endaligned
\end{equation}
where $\nu$ denotes the unit normal vector field of $M$.
Using Cauchy inequality, one has (compared with \eqref{asympnu*})
\begin{equation}\aligned\label{nuEde}
\int_{\mathbf{B}_{\f32\r}\cap M}\min\{|\nu-\xi|,|\nu+\xi|\}<\Psi(\de)\r^n.
\endaligned
\end{equation}
There is a constant $\tau>0$ depending only on $K$ such that $B_{\tau}(\S^{n-2})\cap\widehat{K}=\emptyset$, where $B_{\tau}(\S^{n-2})$ denotes the $\tau$-neighborhood of $\{(0,0,x_3,\cdots,x_{n+1})|\ x_3^2+\cdots+x_{n+1}^2=1\}$ in $\S^n$, i.e., $B_\tau(\S^{n-2})=\{(x_1,\cdots,x_{n+1})\in\S^{n}|\ x_1^2+x_{2}^2<\sin^2\tau\}$.
From \eqref{nuEde}, at least one of $\xi$ and $-\xi$ is in $\widehat{K}\subset\S^n\setminus B_{\tau}(\S^{n-2})$ for the sufficiently small $\de>0$. Then both of $\xi$ and $-\xi$ are in $\S^n\setminus B_{\tau}(\S^{n-2})$.
From Lemma \ref{subhar}, there is a bounded function $\Th$ on $\S^n\setminus B_{\tau}(\S^{n-2})$ with $\Th(\xi)=\Th(-\xi)$
such that
$\Th$ is smooth strictly convex on $\widehat{K}\supset\g(M)$.
Put $v=\Th\circ\g$, then $v$ is a smooth subharmonic function on $M$ with $0<v\le v_{sup}$ for a constant $v_{sup}$ depending only on $n,K$.
Denote $v^*=\Th(\xi)$.
Combining \eqref{vnu} and \eqref{nuEde}, by following the argument of Lemma \ref{Asympv*}, for the sufficiently small $\de>0$ one has
\begin{equation}\aligned\label{vv*4.6}
\int_{\mathbf{B}_{\f32\r}\cap M}|v-v^*|<\Psi(\de)\r^n.
\endaligned
\end{equation}

Now we claim
\begin{equation}\aligned\label{claim}
\sup_{\mathbf{B}_{\r}\cap M}v\le v^*+\Psi(\de).
\endaligned
\end{equation}
Or else, there is a large positive integer $m$ independent of $\de$ such that $\sup_{\mathbf{B}_{\r}\cap M}v> v^*+m^{-\f12}$.
Note that $v_{sup}$ does not depend on $M$. By the monotonicity of $\sup_{\mathbf{B}_{\r}\cap M}v$ on $\r$, there is a constant $\r\le\r_0<\f{5}{4}\r$ such that
\begin{equation}\aligned\label{Br0r4m}
\sup_{\mathbf{B}_{\r_0+\f{\r}{4m}}\cap M}v\le \sup_{\mathbf{B}_{\r_0}\cap M}v+\f{v_{sup}}m.
\endaligned
\end{equation}
Or else, if for all $j\in\{0,\cdots,m-1\}$
\begin{equation}\aligned
\sup_{\mathbf{B}_{\r+\f{j\r}{4m}+\f{\r}{4m}}\cap M}v> \sup_{\mathbf{B}_{\r+\f{j\r}{4m}}\cap M}v+\f{v_{sup}}m,
\endaligned
\end{equation}
then
\begin{equation}\aligned
\sup_{\mathbf{B}_{\f{5\r}{4}}\cap M}v-\sup_{\mathbf{B}_{\r}\cap M}v=\sum_{j=0}^{m-1}\left(\sup_{\mathbf{B}_{\r+\f{j\r}{4m}+\f{\r}{4m}}\cap M}v-\sup_{\mathbf{B}_{\r+\f{j\r}{4m}}\cap M}v\right)>\sum_{j=0}^{m-1}\f{v_{sup}}m=v_{sup}.
\endaligned
\end{equation}
However, it's a contradiction by the definition of $v_{sup}$, and we obtain \eqref{Br0r4m}.
Let $q$ be a point in $\overline{\mathbf{B}_{\r_0}}\cap M$ with $v(q)=\sup_{\mathbf{B}_{\r_0}\cap M}v$. Then $v(q)+\f{v_{sup}}m\ge v$ on $\mathbf{B}_{\f{\r}{4m}}(q)\cap M$ and $v(q)=\sup_{\mathbf{B}_{\r_0}\cap M}v>v^*+m^{-\f12}$.
Hence, $\left(\f{2v_{sup}}m+v(q)-v\right)^{-1}$ is a bounded nonnegative subharmonic function on $\mathbf{B}_{\f{\r}{4m}}(q)\cap M$.
From the mean value inequality \eqref{subf}, one has
\begin{equation}\aligned\label{vvvvv}
\f m{2v_{sup}}\le\f{c_\a m^n}{\r^n}\int_{M\cap\mathbf{B}_{\f{\r}{4m}}(q)}\left(\f{2v_{sup}}m+v(q)-v\right)^{-1},
\endaligned
\end{equation}
where $c_\a$ is a constant depending only on $n,\a$.
Analog to \eqref{vepvep}, combining the choice of $q$ and \eqref{vv*4.6} one has
\begin{equation}\aligned\label{vvvvv*}
\f m{2v_{sup}}\le&\f{c_\a m^n}{\r^n}\int_{M\cap\mathbf{B}_{\f{\r}{4m}}(q)}\left(\f{m^2}{v^2_{sup}}\left|v-v^*\right|+\left(\f{2v_{sup}}m+v(q)-v^*\right)^{-1}\right)\\
\le&\f{c_\a m^{n+2}}{v^2_{sup}}\Psi(\de)+\f{c_\a m^n}{\r^n}\int_{M\cap\mathbf{B}_{\f{\r}{4m}}(q)}m^{\f12}.
\endaligned
\end{equation}
However, \eqref{vvvvv*} fails for the sufficiently large integer $m$ and sufficiently small $m^2\Psi(\de)$. Hence we complete the proof of \eqref{claim}.

By the definition of $\Th$ with the sufficiently large $k>0$, $\g\left(\mathbf{B}_{\r}\cap M\right)\subset\widetilde{E}_{\Psi(\de)}$, where one can find the definition of $\widetilde{E}_t$ in \eqref{EtEt}.
Namely, there is a unit constant vector $(\e_1,\e_2)\in\R^2$ such that $\g\left(\mathbf{B}_{\r}\cap M\right)$ is contained in the closed set $\Om_{\Psi(\de)}$ defined by
$$ \{(x_1,\cdots,x_{n+1})\in\S^{n}|\ x_{1}\e_1+x_2\e_2\ge-\Psi(\de)\}\setminus B_{\tau}(\S^{n-2}).$$
Let $\Om_0$ denote the closed set $\lim_{\de\rightarrow0^+}\Om_{\de}$.
From the proof of Lemma \ref{HessLa***}, there is a positive smooth function $\La$ on a neighborhood of $\Om_{0}$ such that
$\mathrm{Hess}\, \La\le-\left(1-\f1{3n}\right)\La\si_{\S}$ on $\Om_{0}$.
Hence, for the sufficiently small $\de>0$, $\La$ is positive on $\Om_{\Psi(\de)}$ satisfying
\begin{equation}\aligned
\mathrm{Hess}\, \La\le-\left(1-\f1{2n}\right)\La\si_{\S}\qquad \qquad \mathrm{on}\ \Om_{\Psi(\de)},
\endaligned
\end{equation}
and $\La,\f1{\La}$ are bounded on $\Om_{\Psi(\de)}$ by a constant depending on $n,K$, but independent of $\de$.

By following the proof of Theorem \ref{BerOm}, there is a positive constant $C_{K,\a}$ depending only on $n,K,\a$
such that $|A|\le C_{K,\a}/\r$ on $M\cap\mathbf{B}_{3\r/4}$.
Combining \eqref{nuEde}, we get
$$\sup_{M\cap\mathbf{B}_{2\r/3}}\min\{|\nu-\xi|,|\nu+\xi|\}<\Psi(\de).$$
Note that $M$ is a connected smooth manifold.
Hence, there is a smooth function $w$ on a subset of $n$-Euclidean ball $B_{2\r/3}(0)$ in the hyperplane perpendicular to $\xi$ so that $M\cap\mathbf{B}_{2\r/3}(0)$ can be written as a graph of $w$ with $\sup_{B_{3\r/5}(0)}|Dw|<\Psi(\de)$.
Then combining the Schauder estimates of elliptic equations,
we complete the proof.
\end{proof}
\textbf{Remark.} It is interesting to compare Theorem \ref{M1Reg} with the regularity theorem for stable minimal hypersurfaces by Schoen-Simon (Theorem 1 in \cite{SS}.)

\section{Benstein theorem for minimal hypersurfaces}

\begin{lemma}\label{radial line}
Let $M_i$ be a sequence of $n$-dimensional smooth complete oriented embedded minimal hypersurfaces in $\R^{n+1}$ with uniform Euclidean volume growth.
If the Gauss image $\g(M_i)\subset K$, $K$ is a compact set in $\S^n\setminus\overline{\S}^{n-1}_+$, and $M_i$ converges to a nontrivial minimal variety $T\times\R^{n-1}$ in the varifold sense, then $\mathrm{spt}T$ is a line in $\R^2$.
\end{lemma}
\begin{proof}
From the definition of $M_i$, there is a constant $\a>0$ such that
\begin{equation}\aligned\label{ai}
\r^{-n}\mathcal{H}^n\left(M_i\cap\mathbf{B}_{\r}(p)\right)<\a
\endaligned
\end{equation}
for each $i$, $\r>0$, and $p\in\R^{n+1}$.
From \eqref{subv}, integrating by parts infers that there is a constant $c_\a$ depending only on $n$ and $\a$
such that
\begin{equation}\aligned\label{Ai}
\int_{M_i\cap\mathbf{B}_{\r}(p)}|A_i|^2\le c_\a\r^{n-2}
\endaligned
\end{equation}
for all $\r>0$.
Here, $A_i$ is the second fundamental form of $M_i$.
Now we can complete the proof by following the steps in the proof of Theorem 2 of \cite{SS} in which Theorem 1 of \cite{SS} is replaced by Theorem \ref{M1Reg}.
\end{proof}

Recall that $\overline{\S}^{n-1}_+$ is the hemisphere of $\S^{n-1}$ defined by
$$\{(x_1,\cdots,x_{n+1})\in\S^{n}|\ x_1=0,\ x_2\ge0\}.$$
\begin{theorem}\label{main}
Let $M$ be an $n$-dimensional smooth oriented complete embedded minimal hypersurface in $\R^{n+1}$ with Euclidean volume growth. If the image under the Gauss map omits a neighborhood of $\overline{\S}^{n-1}_+$, then $M$ must be an affine hyperplane.
\end{theorem}
\begin{proof}
Let us prove it by contradiction. Assume that $M$ is not affine.
There is a sequence $r_{i}\rightarrow\infty$ such that $r_{i}^{-1}M$ converges to a minimal cone $C$ with integer multiplicity in the varifold sense.
From Corollary \ref{Multione}, the support of $C$ is not a hyperplane. Namely, the singular set of $C$ is not empty.


If there is a singular point $x\in\mathrm{spt}C\setminus\{0\}$, then we blow up the cone $C$ at $x$ and obtain $C'\times\R$, where $C'$ is a minimal cone with a singular point at the origin.
By dimension reduction argument, there is a sequence of smooth minimal hypersurfaces $M_i$, which is obtained from $M$ by scalings and translations, such that $M_i$ converges to a minimal cone $C_*\times\R^{n-k}$ in the varifold sense with $1\le k\le n$, such that $C_*$ is a $k$-minimal cone in $\R^{k+1}$ with the only one singular point at the origin.
From Lemma \ref{radial line}, $C_*$ has dimension $k\ge2$. Since $C_*$ is a $k$-minimal cone in $\R^{k+1}$ with the only one singular point, then $k$ can not be equal to 2. Namely, $k\ge3$.
Let $B^{k+1}_1$ be the unit ball in $\R^{k+1}$ centered at the origin $0^{k+1}$. Here $0^j$ denotes the origin of $\R^{j}$ for each $j\ge1$. Let $\Si=\mathrm{spt}C_*\cap\p B^{k+1}_1$, then $\Si$ is a smooth complete embedded hypersurface in $\p B^{k+1}_1$.

From the assumption of $M$, there is a compact set $K$ contained in the open set $\S^n\setminus \overline{\S}^{n-1}_+$ such that the Gauss image $\g(M)\subset K$.
By the definition of $M_i$, we have $\g(M_i)\subset K$. So the Gauss image of the regular set of the limit $C_*\times\R^{n-k}$ is contained in $K$, namely, $\g((C_*\setminus\{0^{k+1}\})\times\R^{n-k})\subset K$.
Let $\Th$ be a smooth strictly convex function on $K$ defined in Lemma \ref{subhar} (see also \cite{JXY}).
Then from \eqref{subv} the function $v=\Th\circ\g$ satisfies
\begin{equation}\aligned\label{DeC*Rv}
\De_{C_*\times\R^{n-k}} v\ge \k|A_{C_*\times\R^{n-k}}|^2
\endaligned
\end{equation}
on the regular set of $C_*\times\R^{n-k}$ for some positive constant $\k>0$,
where $\De_{C_*\times\R^{n-k}}$, $A_{C_*\times\R^{n-k}}$ are the Laplacian and the second fundamental form of $C_*\times\R^{n-k}$ on the regular set of $C_*\times\R^{n-k}$, respectively. In particular, $v$ is uniformly bounded.

Let $\De_{C_*}$, $A_{C_*}$ be the Laplacian and the second fundamental form of $C_*$ on the regular set of $C_*$, respectively.
Through restricting $v$ on spt$C_*\times\{0^{n-k}\}$, from \eqref{DeC*Rv} one has
\begin{equation}\aligned\label{DeC*v}
\De_{C_*} v\ge \k|A_{C_*}|^2
\endaligned
\end{equation}
on spt$C_*\setminus\{0^{k+1}\}$.
Let $\De_{\Si}$, $A_{\Si}$ be the Laplacian and the second fundamental form of $\Si$, respectively. Then \eqref{DeC*v} infers
\begin{equation}\aligned\label{DeSiv}
\De_{\Si} v\ge \k|A_{\Si}|^2\qquad \mathrm{on}\ \ \Si.
\endaligned
\end{equation}
The above inequality contradicts to the maximum principle. This is sufficient to complete the proof.
\end{proof}




\bibliographystyle{amsplain}

\end{document}